\documentclass[12pt]{amsart}
\usepackage{amsfonts, amsbsy, amsmath, amssymb}
\usepackage{xcolor}
\usepackage{placeins}

\input prepictex.tex
\input pictex.tex
\input postpictex.tex

\usepackage{graphicx}
\newtheorem{thm}{Theorem}[section]
\newtheorem{lem}[thm]{Lemma}
\newtheorem{cor}[thm]{Corollary}
\newtheorem{prop}[thm]{Proposition}
\newtheorem{exmp}[thm]{Example}

\newtheorem{conj}[thm]{Conjecture}

\newtheorem{rmk}[thm]{Remark}

\newtheorem{ques}[thm]{Question}
\newtheorem{thm-con}[thm]{Theorem-Conjecture}
\numberwithin{equation}{section}

\theoremstyle{definition}
\newtheorem{defn}[thm]{Definition}

\newcommand{\x}{{\tt X}}
\newcommand{\y}{{\tt Y}}
\newcommand{\z}{{\tt Z}}
\newcommand{\f}{\Bbb F}

\newcommand{\ar}[2]{\arrow <4pt> [0.3, 0.67] from #1 to #2}

\title[Polynomial $g_{n,q}$]{A further study of polynomial $g_{n,q}$ over finite fields}

\begin{document}

\author[N. Fernando]{Neranga Fernando} 
\address{Department of Mathematics,
Knox College, Galesburg, IL 61401, USA}
\email{nfernando@knox.edu}

\author[Bhitali Kousik]{Bhitali Kousik}
\address{Department of Mathematical Sciences, Tezpur University, Tezpur, Assam 784028, India}
\email{msp24013@tezu.ac.in}

\maketitle

\begin{abstract}
Let $n\geq 0$ be an integer and $q$ a prime power. The polynomial $g_{n,q}$ was introduced in \cite{Hou-JCTA-2011} with the purpose of finding new classes of permutation polynomials over finite fields. We investigate the permutation behaviour of the polynomial $g_{n,q}(\x)$ over finite fields of even characteristic. We introduce the multivariate case of the polynomial $g_{n,q}$, and study the permutation polynomials in several variables and local permutation polynomials resulting from the polynomials $g_{n,q}(\x_1,\x_2,\ldots , \x_k)$. We also present several new identities of $g_{n,q}(\x)$, and present some open questions on the permutation property of $g_{n,q}(\x)$. 
\end{abstract}

\section{Introduction}\label{S1}

Let $p$ be a prime, $e$ and $m$ be positive integers, $q=p^m$, and $\mathbb{F}_q$ the finite field with $q$ elements. A polynomial $f(\x)\in \mathbb{F}_q[\x]$ is a called a \textit{permutation polynomial} if the associated mapping from $\mathbb{F}_q$ to $\mathbb{F}_q$ is a permutation. It is a well-known fact that every mapping from $\mathbb{F}_q$ to $\mathbb{F}_q$ can be uniquely represented by a polynomial of degree lees than $q$, and the unique polynomial can be found by Lagrange Interpolation. Finding new classes of permutation polynomials over finite fields has been a hot topic since the 1970s due to their applications in the areas of cryptography, coding theory, combinatorics, finite geometry, and computer science. 

The study of permutation polynomials over finite fields has a long history. It was Charles Hermite who studied permutation polynomials over finite prime fields in the 1860s. Leonard Eugene Dickson studied permutation polynomials over finite arbitrary fields while he was a PhD student in the 1890s. Issai Schur later named these polynomials in Dickson's honor. Dickson polynomials have played a pivotal role in the area of permutation polynomials over finite fields, and the permutation property of the Dickson polynomials is completely known. 

In 2009, Xiang-dong Hou, Gary Mullen, James Sellers and Joseph Yucas introduced the twin of Dickson polynomials, reversed Dickson polynomials, by interchanging the roles of the variable and the parameter of Dickson polynomials. Polynomial $g_{n,q}$ was introduced by Xiang-dong Hou in \cite{Hou-JCTA-2011} as a $q$-ary version of the reversed Dickson polynomial, but its permutation behaviour was first studied in \cite{Hou-FFA-2012}. In \cite{Hou-JCTA-2011}, the author showed that for each integer $n \geq 0$, there exists a unique polynomial $g_{n,q} \in \Bbb F_p[\x]$ such that
\begin{equation}\label{AB1}
\displaystyle\sum_{a \hspace{0.1cm}\in \hspace{0.1cm} \mathbb{F}_{q}} ({\x}+a)^{n} = g_{n,q}({\x}^{q} - {\x}).
\end{equation}
The explicit form of $g_{n,q}$ is given by Waring's formula
\begin{equation}\label{AB2}
g_{n,q}({\x}) = \sum_{\frac{n}{q}\leq l\leq \frac{n}{q-1}}\frac{n}{l}\dbinom{l}{n-l(q-1)}{\x}^{n-l(q-1)}.
\end{equation}

If $g_{n,q}$ is a PP of $\Bbb F_{q^{e}}$, we say that the triple $(n,e;q)$ is \textit{desirable}. Since 2011, several articles have appeared in the literature on polynomial $g_{n,q}$: \cite{Fernando-2019}, \cite{Fernando-Hou-FFA-2015}, \cite{Fernando-Hou-Lappano-FFA-2013}, \cite{Fernando-Hou-Lappano-DM-2014}, and \cite{Hou13}, to name a few. The authors in \cite{OS21} showed that the generalized almost perfect nonlinear (GAPN) functions  are related to the polynomial $g_{n,q}$. In fact, they showed that If $g_{n,p}(\x)$ is a permutation of $\mathbb{F}_{p^e}$, then $f(\x)=\x^n$ is a GAPN function. 

Even though the permutation behaviour of the polynomial $g_{n,q}$ has been studied extensively, multivariate case of the polynomial has not appeared in the literature. In this paper, we introduce the multivariate case of the polynomial $g_{n,q}$ and derive an explicit formula for the polynomial in several variables. We first show that $$g_{n,q}(\x,\y)=g_{n,q}(\x)\circ e_1(\x,\y),$$ where $e_1(\x,\y)$ is the elementary symmetric polynomials in two variables, which is immediately followed by the generalization to several variables: $$g_{n,q}(\x_1,\x_2,\ldots ,\x_k)=g_{n,q}(\x)\circ e_1(\x_1,\x_2,\ldots ,\x_k),$$
where $e_1(\x_1,\x_2,\ldots ,\x_k)$ is the elementary symmetric polynomial in $k$ variables. 

Let $\mathbb{F}_q^k$ be the $k$-fold cartesian product of $\mathbb{F}_q$, where $k\geq 1$ is an integer. Let $\overline{\x}=(\x_1,\ldots, \x_k)$, $\overline{\x}_i=(\x_1,\ldots, \x_{i-1},\x_{i+1},\ldots, \x_k)$, and $\mathbb{F}_q[\overline{\x}]$ denote the ring of polynomials in $k$ variables over $\mathbb{F}_q$. It is well-known that any map from $\mathbb{F}_q^k$ to $\mathbb{F}_q$ can be uniquely represented as a polynomial $f\in \mathbb{F}_q[\overline{\x}]$ such that $\text{deg}_{\x_i}(f)<q$ for all $i=1,\ldots, k$, where $\text{deg}_{\x_i}(f)$ is the degree of $f$ as a polynomial in the variable $\x_i$ with coefficients in the polynomial ring $\mathbb{F}_q[\overline{\x}_i]$. \vskip 0.1in 

\begin{defn}
A polynomial $f\in \mathbb{F}_q[\x_1,\ldots, \x_k]$ is a \textit{permutation polynomial} in $k$ variables over $\mathbb{F}_q$ if the equation $f(\x_1,\ldots, \x_k)=\alpha$ has exactly $q^{k-1}$ solutions in $\mathbb{F}_q^k$ for each $\alpha \in \mathbb{F}_q$.    
\end{defn}

\begin{defn}
A polynomial $f\in \mathbb{F}_q[\overline{\x}]$ is called a \textit{local permutation polynomial} (or LPP) if for each $i$, $1\leq i\leq k$, the polynomial $f(a_1,\ldots, a_{i-1},\x_i,a_{i+1}, \ldots, a_k)$ is a permutation polynomial in $\mathbb{F}_q[\x_i]$, for all choices of $\overline{a}_i\in \mathbb{F}_q^{k-1}$. 
\end{defn}

Local permutation polynomials have been studied by many due to their applications in the areas of cryptography, coding theory and latin squares. Any LPP is a permutation polynomial, but the converse is not true in general. For example, the polynomial $\x_1^{q-1}+\x_2$ is a permutation polynomial over $\mathbb{F}_q$, but not a local permutation polynomial. To the best of our knowledge, local permutation polynomials were first studied by Gary L. Mullen in \cite{Mullen-2} and \cite{Mullen-3}. He gave necessary and sufficient conditions for polynomials in two and three variables to be local permutations polynomials over a finite prime field $\mathbb{F}_p$ in \cite{Mullen-2} and \cite{Mullen-3}, respectively. Those conditions were expressed in terms of the coefficients of the polynomial. \vskip 0.1in 

For two variables, it was shown in \cite{DHK} that the degree of an LLP in $\mathbb{F}_q[x,y]$ is bounded above by $2(q-2)$. In \cite{GJ}, the authors showed that the result is true for several variables. We refer the reader to \cite{DHK}, \cite{GJ}, \cite{GJ-1}, \cite{GJ-2}, \cite{Mullen-2}, and \cite{Mullen-3} for detailed studies on local permutation polynomials. 

In order to investigate the permutation polynomials in several variables and local permutation polynomials arising from $g_{n,q}(\x_1,\x_2,\ldots ,\x_k)$, we first derive some results that are analogous to that of the univariate case. We then show that studying the permutation behaviour or the local permutation behaviour of the multivarite $g$-polynomial is equivalent to studying the permutation behaviour of its univariate case. Because of this reason, we take a new approach and investigate the permutation and the local permutation behaviour of the polynomial
$$g_{n,q}^{(\ell)}(\x_1,\,\x_2,\dots,\,\x_k)=g_{n,q}(x)\circ\sigma_{\ell}(\x_1,\,\x_2,\dots,\,\x_k),$$ where $\sigma_{\ell}=\sigma_{\ell}(\x_1,\,\x_2,\dots,\,\x_k)=\x_1^{\ell}+\x_2^{\ell}+\cdots +\x_k^{\ell},$ 
where $\ell\geq 1.$

We first show that if $\gcd(\ell,\,q^e-1)=1$, then $g_{n,q}^{(\ell)}\in\mathbb{F}_{q^e}[\bar{\x}]$ is a (local) permutation polynomial over $\mathbb{F}_{q^e}$ if and only if $g_{n,q}(\x)$ is a PP over $\mathbb{F}_{q^e}$. We then show that  $g_{n,q}^{(\ell)}$ can never be an LPP over $\mathbb{F}_{q^e}$ whenever $\gcd(\ell,\,q^e-1)\neq 1$.  We also conjecture that the polynomial $g_{n,q}^{(\ell)}\in \mathbb{F}_{q^e}[\x_1,\x_2,\ldots ,\x_k]$ is a PP over $\mathbb{F}_{q^e}$ if and only if $\text{gcd}(\ell,q^e-1)=1$ and $g_{n,q}$ is a PP over $\mathbb{F}_{q^e}$. 

 We would like to point out to the reader that even though the multivariate polynomials resulting from $g_{n,q}^{(\ell)}(\x_1,\,\x_2,\dots,\,\x_k)$ may not necessarily belong to the family of polynomial $g_{n,q}(\x_1,\,\x_2,\dots,\,\x_k)$, we believe that our results on the permutation property and the local permutation property of the polynomials $g_{n,q}^{(\ell)}(\x_1,\,\x_2,\dots,\,\x_k)$ would complement the area of LPPs and PPs in several variables over finite fields. 

The organization of the paper is as follows. 

Identities of the polynomial $g_{n.q}$ have played an important role in deriving polynomials for explaining their permutation behaviour. We present twelve new identities of the polynomial $g_{n,q}$ in Section 2. Section 3 of this paper is a continuation of \cite{Fernando-Hou-FFA-2015} and \cite{Fernando-2019}. Table 1 in \cite{Fernando-Hou-FFA-2015} contains all desirable triples $(n,e:4)$ with $e\leq 6$ and $w_q(n)>4$, where $w_q(n)$ stands for the base 4 weight of $n$. One of the unexplained desirable triples in \cite[Table~1]{Fernando-Hou-FFA-2015} was explained in \cite{Fernando-2019}. Section 3 is an attempt to answer several unexplained cases when $q=4$. 

In Section 4, we derive the multivariate case of the polynomial $g_{n,q}$ and explore some results analogous to the univariate case. In Section 5, we study permutation polynomials in several variables and local permutation polynomials associated with polynomial $g_{n,q}(\x_1,\x_2,\ldots, \x_k)$ over finite fields. 

In Section 6, we present eight open queastions regarding the permutation behaviour of the polynomial $g_{n,q}(\x)$ over finite fields of even characteristic. We conclude the paper with an updated table that contains all desirable triples $(n,e:4)$ with $e\leq 6$ and $w_q(n)>4$, where $w_q(n)$ denotes the base 4 weight of $n$. 

\section{Some new identities of the polynomial $g_{n,q}(\x)$}

Let $k\geq 0$ be an integer. For a fixed $q,$ we define $S_k=\x+\x^q+\x^{q^2}+\cdots +\x^{q^{k-1}}$ for $k\geq 1$ and $S_0=0.$ Note that, $\text{Tr}_{q^e/q}=S_e,$ where $\text{Tr}_{q^e/q}$ is the trace function  from $\mathbb{F}_{q^e}$ to $\mathbb{F}_q.$

The following three results from the literature relate the polynomial $g_{n,q}$ to $S_k$, and they will be useful in deriving further identities of $g_{n,q}.$

\begin{prop}\text{\rm(\cite[Eqn 4.1]{Hou-FFA-2012})}\label{lem 2.2} For integers $l,\,k\geq 0,$ we have
$$g_{l+q^k,q}=g_{l+1,q}+S_k\,g_{l,q}.$$
\end{prop}
In Section 4, we will explain the multivariate case of Proposition~\ref{lem 2.2}. 
\begin{thm}\text{\rm(\cite[Theorem 6.1]{Fernando-Hou-Lappano-FFA-2013})}\label{T6.1a}
Let $q\ge 4$ be even, and let
\[
n=1+q^{a_1}+q^{b_1}+\cdots+q^{a_{q/2}}+q^{b_{q/2}},
\]
where $a_i,b_i\ge 0$ are integers.
Then
\[
g_{n,q}=\sum_iS_{a_i}S_{b_i}+\sum_{i<j}(S_{a_i}+S_{b_i})(S_{a_j}+S_{b_j}).
\]
\end{thm}

\begin{lem}\text{\rm(\cite[Lemma 6.5]{Fernando-Hou-Lappano-FFA-2013})}
Let $n=(q-1)q^a+(q-1)q^b$, where $a,b\ge 0$. Then
\[
g_{n,q}=-1-(S_b-S_a)^{q-1}.
\]
\end{lem}

For the rest of the section, we list twelve new identities of the polynomial $g_{n,q}(x)$. 

\begin{lem}\label{L310}
Let $q>2$ be even and $n=1\cdot q^0+(q-1)\cdot q^a+(q-1)\cdot q^b$, where $a,b\geq 0$. Then 
$$g_{1\cdot q^0+(q-1)\cdot q^a+(q-1)\cdot q^b}=S_{b-a+1}^{q^a}+S_b+S_b\,(S_b-S_a)^{q-1}.$$
\end{lem} 

\begin{proof}
We have 
$$g_{(q-1)\cdot q^a+q^{b+1}}=g_{1\cdot q^0+(q-1)\cdot q^a+(q-1)\cdot q^b}+S_b\cdot g_{(q-1)\cdot q^a+(q-1)\cdot q^b},$$

which implies
$$\x^{q^a}+\x^{q^{a+1}}+\cdots +\x^{q^{b-1}}+\x^{q^b}=g_{1\cdot q^0+(q-1)\cdot q^a+(q-1)\cdot q^b}+S_b\,(-1-(S_b-S_a)^{q-1}).$$

Thus we have 

\[
\begin{split}
g_{1\cdot q^0+(q-1)\cdot q^a+(q-1)\cdot q^b}&=(\x+\x^q+\x^{q^2}+\cdots +\x^{q^{b-a}})^{q^a}-S_b\,(-1-(S_b-S_a)^{q-1})\cr
&=S_{b-a+1}^{q^a}+S_b+S_b\,(S_b-S_a)^{q-1}
\end{split}
\]

This completes the proof.

\end{proof} 

\begin{lem}\label{lem 3.2}
Let $a,b\geq 0$ be integers. Then 
$$g_{(q-1)\cdot q^a+(q-2)\cdot q^b}=-(S_b-S_a)^{q-2}.$$
\end{lem} 

\begin{proof} 
\[
\begin{split}
(S_b-S_a)g_{(q-1)\cdot q^a+(q-2)\cdot q^b}&=g_{(q-1)\cdot q^a+(q-2)\cdot q^b+q^b}-g_{(q-1)\cdot q^a+(q-2)\cdot q^b+q^a}\cr
&=g_{(q-1)\cdot q^a+(q-1)\cdot q^b}-g_{q^{a+1}+(q-2)\cdot q^b}\cr
&=-1-(S_b-S_a)^{q-1}-(-1)\cr
\end{split} 
\]

This completes the proof.

\end{proof} 

In the following lemma, we list seven more idendities of the polynomial $g_{n,q}$.

\begin{lem}\label{lem 3.3}
Let $a,b\geq 0$ be integers. Then
\begin{enumerate}
    \item 
    $g_{(q-2)\cdot q^a+(q-1)\cdot q^b}=(S_b-S_a)^{q-2},$
    \item 
    $g_{(q-2)q^a+(q-2)q^b}=2(S_b-S_a)^{q-3},$
    \item 
    $g_{(q-1)q^a+(q-3)q^b}=-(S_b-S_a)^{q-3},$
    \item 
    $g_{(q-3)q^a+(q-1)q^b}=-(S_b-S_a)^{q-3},$
    \item 
    $g_{(q-3)q^a+(q-2)q^b}=-3(S_b-S_a)^{q-4},$
    \item 
    $g_{(q-3)q^a+(q-2)q^b}=3(S_b-S_a)^{q-4},$
    \item 
    $g_{(q-3)q^a+(q-3)q^b}=-6(S_b-S_a)^{q-5}.$
\end{enumerate}
\end{lem}

We omit the proofs of the identities of Lemma~\ref{lem 3.3} as they are similar to that of Lemma~\ref{lem 3.2}. 

Since we believe that any polynomial $g_{n,q}$ enthusiast will find the identities given in Lemma~\ref{lem 3.4}, Lemma~\ref{lem 3.5}, and Lemma~\ref{lem 3.6} useful in investigating the permutation behaviour as well as the irreducibility of the polynomials $g_{n.,q}$, we present them for completeness. 

\begin{lem}\label{lem 3.4}
    Let $l=(q-3)q^0+(q-3)q^1.$ Then $$g_{l+q^2}=3S_1^{q-4}-6S_2S_1^{q-5},$$ and $$g_{l+2q^2}=-S_1^{q-3}+6S_2S_1^{q-4}-6S_2^2S_1^{q-5}.$$  
\end{lem}
\begin{proof} The result follows by applying the identities obtained in the previous lemma.
\end{proof}
\begin{lem}\label{lem 3.5}
    Let $n=(q-3)q^0+(q-3)q^1+3q^2.$ Then $$g_{n,q}=-3S_2S_1^{q-3}+9S_2^2S_1^{q-4}-6S_2^3S_1^{q-5}. $$
\end{lem}
\begin{proof} Let $l=(q-3)q^0+(q-3)q^1.$ We have
    \begin{eqnarray*}
g_{l+3q^2}&=&g_{l+2q^2+1}+S_2 g_{l+2q^2}\\
&=&g_{l+2q^2+q}-S_1g_{l+2q^2}+S_2g_{l+2q^2}\\
&=&g_{l+2q^2}(S_2-S_1)+g_{(q-3)q^0+(q-2)q^1+2q^2}.
\end{eqnarray*}
We now compute $g_{j+2q^2},$ where $j=(q-3)q^0+(q-2)q^1.$ To do so, we first determine $g_{j+q^2}.$

$$g_{j+q^2}=g_{j+1}+S_2 g_j=2S_1^{q-3}-3S_2S_1^{q-4}.$$
Now,
\[
\begin{split}
    g_{j+2q^2}&=g_{j+q^2+1}+S_2 g_{j+q^2}=g_{j+2}+S_2g_{j+1}+S_2g_{j+q^2}\cr
&=-S_1^{q-2}+2S_2S_1^{q-3}+2S_2S_1^{q-3}-3S_2^2S_1^{q-4}\cr
&=-S_1^{q-2}+4S_2S_1^{q-3}-3S_2^2S_1^{q-4}.
\end{split}
\]
Thus by using the value of $g_{l+2q^2}$ from Lemma~\ref{lem 3.4}, we have,
\[
\begin{split}
  g_{l+3q^2}&=(-S_1^{q-3}+6S_2S_1^{q-4}-6S_2^2S_1^{q-5})(S_2-S_1)-S_1^{q-2}+4S_2S_1^{q-3}-3S_2^2S_1^{q-4}\cr
&=-3S_2S_1^{q-3}+9S_2^2S_1^{q-4}-6S_2^3S_1^{q-5}.  
\end{split}
\]

\end{proof}
For convenience, we write $A=g_{j+2q^2}$ and $B=g_{l+3q^2}$, with these expressions coming from the proof of the preceding lemma.
\begin{lem}\label{lem 3.6}
    Let $n=(q-3)q^0+(q-3)q^1+4q^2.$ Then $g_{n,q}=-1+6S_2^2S_1^{q-5}(S_2^2+2S_2S_1-S_1^2).$
\end{lem}
\begin{proof}
    Let $l=(q-3)q^0+(q-3)q^1.$ We have
    \[
    \begin{split}
        g_{l+4q^2}&=g_{l+3q^2+1}+S_2g_{l+3q^2}\cr
&=g_{l+3q^2+q^1}-S_1g_{l+3q^2}+S_2g_{l+3q^2}\cr
&=g_{l+q^1+2q^2+1}+S_2g_{l+q^1+2q^2}+(S_2-S_1)g_{l+3q^2}\cr
&=g_{(q-3)q^0+(q-1)q^1+2q^2}+(S_2-S_1)(A+B)
    \end{split}
    \]
Next, we determine $g_{j+q^1}$ and $g_{j+q^2},$ where $j=(q-3)q^0+(q-1)q^1.$
$$g_{j+q^2}=g_{j+1}+S_2g_j=S_1^{q-2}-S_2S_1^{q-3},$$ and
\[
\begin{split}
    g_{j+2q^2}&=g_{j+q^2+1}+S_2g_{j+q^2}=g_{j+2}+S_2g_{j+1}+S_2g_{j+q^2}\cr
    &=-1-S_1^{q-1}+2S_2S_1^{q-2}-S_2^2S_1^{q-3}.
\end{split}
\]

Also, $A+B=-S_1^{q-2}+S_2S_1^{q-3}+6S_2^2S_1^{q-4}-6S_2^3S_1^{q-5}.$ Hence, after simplification, we have
\[
\begin{split}
   g_{l+4q^2}&=-1-S_1^{q-1}+2S_2S_1^{q-2}-S_2^2S_1^{q-3}\cr
& +(S_2-S_1)(-S_1^{q-2}+S_2S_1^{q-3}+6S_2^2S_1^{q-4}-6S_2^3S_1^{q-5})\cr
&=-1+6S_2^2S_1^{q-5}(S_2^2+2S_2S_1-S_1^2). 
\end{split}
\]

 Applying $S_2=\x+\x^q$ and $S_1=\x$ to the above expression, we obtain
$$g_{l+4q^2,q}=-1-6\x^{5q-5}-12\x^{4q-4}-6\x^{3q-3}.$$
\end{proof}

\section{PPs over finite fields of even characteristic}\label{S3}

In this section, we investigate permutation polynomials over finite fields of even characteristic resulting from polynomials $g_{n,q}(\x)$. Let $\overline{\mathbb{F}}_p$ be the algebraic closure of $\mathbb{F}_p$ and $f=\sum_{i=0}^na_i\x^{q^i}\in\overline\f_p[\x]$ be a $q$-linearized polynomial. The conventional associate of $f$ is the polynomial $\widetilde f=\sum_{i=0}^na_ix^i\in\overline\f_p[\x]$. We first present a result from \cite{Fernando-Hou-FFA-2015} that will play a key role in our results throughout this section. 



\begin{prop}\label{P3.1}\text{\rm(\cite{Fernando-Hou-FFA-2015})}
Let $m$ and $e$ be positive integers, $r$ a prime power and $q=r^m$. 
A polynomial $f\in\f_{q^e}[\x]$ is a PP of $\f_{q^e}$ if the following conditions are all satisfied.
\begin{itemize}
  \item [(i)] There exists a PP $\bar f\in\f_q[\x]$ of $\f_q$ such that the diagram
\[
\beginpicture
\setcoordinatesystem units <3mm,3mm> point at 0 0

\ar{1 0}{5 0}
\ar{1 6}{5 6}
\ar{0 5}{0 1}
\ar{6 5}{6 1}
\put {$\f_{q^e}$} at 0 6
\put {$\kern1mm\f_{q^e}$} at 6 6
\put {$\f_q$} at 0 0
\put {$\f_q$} at 6 0
\put {$\scriptstyle f$} [b] at 3 6.2 
\put {$\scriptstyle \bar f$} [t] at 3 -0.2 
\put {$\scriptstyle S_e$} [r] at -0.2 3
\put {$\scriptstyle S_e$} [l] at 6.2 3
\endpicture
\]  
commutes.

\item[(ii)] For each $c\in\f_q$, there exist $q$-linearized polynomials $f_{c,i}\in\f_r[\x]$ and $a_{c,i}\in\f_q$, $0\le i\le m-1$, and $b_c\in\f_{q^e}$ such that
\begin{equation}\label{3.1}
f(x)=f_c(x)+b_c\quad\text{for all}\ x\in S^{-1}(c),
\end{equation}
where 
\begin{equation}\label{3.2}
f_c=\sum_{i=0}^{m-1}a_{c,i}f_{c,i}^{r^i}.
\end{equation}

\item[(iii)] For each $c\in\f_q$,
\begin{equation}\label{3.3}
\text{\rm gcd}\bigl(\det A_c,\ (\x^e-1)/(\x-1)\bigr)=1,
\end{equation}
where
\begin{equation}\label{3.4}
A_c=\left[
\begin{matrix}
a_{c,0}\widetilde f_{c,0} & a_{c,1}\widetilde f_{c,1} &\cdots& a_{c,m-1}\widetilde f_{c,m-1}\cr
a_{c,m-1}^r\widetilde f_{c,m-1}\x & a_{c,0}^r\widetilde f_{c,0} &\cdots& a_{c,m-2}^r\widetilde f_{c,m-2}\cr
\vdots&\vdots&&\vdots\cr
a_{c,1}^{r^{m-1}}\widetilde f_{c,1}\x & a_{c,2}^{r^{m-1}}\widetilde f_{c,2}\x &\cdots& a_{c,0}^{r^{m-1}}\widetilde f_{c,0}
\end{matrix}\right].
\end{equation}
\end{itemize}
\end{prop}

The following two theorems yield two new families of permutation polynomials over finite fields of even characteristic which would also explain two unexplained cases of  \cite[Table~1]{Fernando-Hou-FFA-2015}. 

\begin{thm}\label{T3.25}
Let $q=4$ and $n=1\cdot q^0+2\cdot q^1+2\cdot q^2+1\cdot q^e+2\cdot q^{e+2}$. Let $e>2$ be even such that $e\not\equiv 0\pmod{10}$. Then $g_{n,q}=\x^2+S_e^2+S_3S_e+\x^{2q}\,S_e^3$ is a PP of $\mathbb{F}_{q^e}$.
\end{thm}  

\begin{proof}
It can be shown that 
$$g_{n,q}=S_eS_3+(\x+S_e^3\x^{q})^2+S_e^2.$$

We show that the conditions (i)-(iii) in Proposition~\ref{P3.1} are satisfied when $r=2$, $m=2$, and $q=4$. Condition (i) is satisfied with $\overline{f}=\x^2$. For each $c\in \mathbb{F}_q$ and $\x\in S_e^{-1}(c)$, 
$$f(\x)=cS_3+(\x+c^3\x^{q})^2+c^2,$$
Hence (ii) is satisfied with
\[
f_{c,0}=S_3,\,\,\,f_{c,1}=\x+c^3\x^q,\ a_{c,0}=c,\ a_{c,1}=1,\ b_c=c^2.
\]
We have 
\[
\det A_c=\left|\begin{matrix} c(1+\x+\x^{2})&(1+c^3X)\cr \x(1+c^3\x)& c^2(1+\x+\x^{2})\end{matrix}\right|=c^3(1+\x+\x^2)^2-\x(1+c^3\x^2).
\]
Since $e\not\equiv 0\pmod{10}$, $\text{gcd}(\det A_c,\ (\x^e-1)/(\x-1))=1$, and hence (iii) is also satisfied.

\end{proof}

\vskip 0.2in

\begin{exmp} 
Let $q=4$ and $e=6$. Then $g_{135209,q}$ is a PP of $\mathbb{F}_{q^e}$. We have 
$$135209=1\cdot q^0+2\cdot q^1+2\cdot q^2+1\cdot q^6+2\cdot q^{8}.$$ It can be shown that $g_{n,q}=S_6S_3+(\x+S_6^3\x^{4})^2+S_6^2.$ For each $c\in \mathbb{F}_q$ and $\x\in S_e^{-1}(c)$, $g_{n,q}=cS_3+(\x+c^3\x^{4})^2+c^2$, and the result follows from the fact that
$e\not\equiv 0\pmod{10}$ and $\text{gcd}(c^3(1+\x+\x^2)^2-\x(1+c^3\x^2),\ (\x^6-1)/(\x-1))=1$. 
\end{exmp} 

\vskip 0.2in 


\vskip 0.2in

\begin{thm}\label{T3.27}
Let $q=4$ and $n=1\cdot q^0+2\cdot q^2+2\cdot q^{e-1}+1\cdot q^e+2\cdot q^{e+2}$. Let $e>3$ be even such that $e\not\equiv 0\pmod{10}$. Then $g_{n,q}$ is a PP of $\mathbb{F}_{q^e}$.
\end{thm}  

\begin{proof}

It can be shown that 

$$g_{n,q}=S_eS_3+\Big(S_e^3(\x+\x^q)+(S_e^3+1)\x^{q^{e-1}}\Big)^2+S_e^2.$$

We show that the conditions (i)-(iii) in Proposition~\ref{P3.1} are satisfied when $r=2$, $m=2$, and $q=4$. Condition (i) is satisfied with $\overline{f}=X^2$. For each $c\in \mathbb{F}_q$ and $x\in S_e^{-1}(c)$, 
$$f(x)=cS_3+\Big(c^3(\x+\x^q)+(c^3+1)\x^{q^{e-1}}\Big)^2+c^2,$$
Hence (ii) is satisfied with
\[
f_{c,0}=S_3,\,\,\,f_{c,1}=c^3(\x+\x^q)+(c^3+1)\x^{q^{e-1}},\ a_{c,0}=c,\ a_{c,1}=1,\ b_c=c^2.
\]
We have 
\[
\begin{split}
\det A_c&=\left|\begin{matrix} c(1+\x+\x^{2})& c^3(1+\x)+(c^3+1)\x^{e-1}\cr \x (c^3(1+\x)+(c^3+1)\x^{e-1})& c^2(1+\x+\x^{2})\end{matrix}\right|\cr
&=c^3(1+\x+\x^2)^2-\Big(c^3(1+\x)+(c^3+1)\x^{e-1}\Big)\cdot \Big(\x (c^3(1+\x)+(c^3+1)\x^{e-1})\Big).
\end{split}
\]
Since $e\not\equiv 0\pmod{10}$, $\text{gcd}(\det A_c,\ (\x^e-1)/(\x-1))=1$, and hence (iii) is also satisfied.

\end{proof}

\vskip 0.2in

\begin{exmp} 
Let $q=4$ and $e=6$. Then $g_{137249,q}$ is a PP of $\mathbb{F}_{q^e}$.  We have 
$$137249=1\cdot q^0+2\cdot q^2+2\cdot q^{e-1}+1\cdot q^e+2\cdot q^{e+2}.$$ It can be shown that $g_{n,q}=S_6S_3+\Big(S_6^3(\x+\x^4)+(S_6^3+1)\x^{4^{5}}\Big)^2+S_6^2.$ For each $c\in \mathbb{F}_q$ and $\x\in S_e^{-1}(c)$, $g_{n,q}=cS_3+\Big(c^3(\x+\x^4)+(c^3+1)\x^{4^{5}}\Big)^2+c^2.$, and the result follows from the fact that
$e\not\equiv 0\pmod{10}$ and $$\text{gcd}(c^3(1+\x+\x^2)^2-(c^3(1+\x)+(c^3+1)\x^{5})\cdot (\x (c^3(1+\x)+(c^3+1)\x^{5})),\ (\x^6-1)/(\x-1))=1.$$
\end{exmp} 

The following proposition explains the entry $n=4289$ when $e=5$ in the table. 

\begin{prop}\label{T4.5}
Let $q=4$ and $n=1+(q-1)\cdot q^{e-2}+1\cdot q^{e+1}$, where $e>2$. 
Then $g_{n,q}=S_eS_3^{q^{e-2}}+(S_2S_3)^{q^{e-2}}$. In particular, $g_{n,q}$ is a PP of $\mathbb{F}_{q^5}$. 
\end{prop} 

\begin{proof}

Let $n=1+(q-1)\cdot q^{e-2}+1\cdot q^{e+1}$. Then it can be shown that $$g_{n,q}=S_3^{q^{e-2}}\,S_{e-2}.$$ 
Let $e=5$. Then we have $g_{n,q}=S_3^{q^{3}+1}$, which is a PP since $\text{gcd}(q^3+1,q^5-1)=1$.

\end{proof}

\begin{ques}
Is Proposition~\ref{T4.5} a sporadic case? If it is not, find conditions on $e$ for which the polynomial $g_{n,q}$ is a PP of $\mathbb{F}_{4^e}.$   
\end{ques}

\section{Multivariate Case} 

In this section, we introduce the bivariate case of the polynomial $g_{n,q}$.

\subsection{Bivariate case}

In $\mathbb{F}_q[\x,\y]$, we have $$\displaystyle{(\x+\y)^q-(\x+\y)=\prod_{a\in \mathbb{F}_q}(\x+\y+a)}.$$ 

Let ${\tt t}$ be another indeterminate and replace $\x+\y$ by ${\tt t}+\x+\y$ to obtain 
$$\displaystyle{({\tt t}+\x+\y)^q-({\tt t}+\x+\y)=\prod_{a\in \mathbb{F}_q}({\tt t}+\x+\y+a)}.$$

The left-hand side is equal to ${\tt t}^q-{\tt t}+(\x^q-\x)+(\y^q-\y)$. 

The product $\displaystyle{\prod_{a\in \mathbb{F}_q}({\tt t}+\x+\y+a)}$ can be written in terms of the elementary symmetric polynomials (as described in Preliminaries) as follows:

$$\displaystyle{\prod_{a\in \mathbb{F}_q}({\tt t}+\x+\y+a)}=\sum_{k=0}^{q}\,\Big(\sigma_k(\x+\y+a)_{a\in \mathbb{F}_q}\Big){\tt t}^{q-k}.$$

We have 

$${\tt t}^q-{\tt t}+(\x^q-\x)+(\y^q-\y)=\sum_{k=0}^{q}\,\sigma_k\Big((\x+\y+a)_{a\in \mathbb{F}_q}\Big)\,{\tt t}^{q-k}.$$

By comparing the coefficients of ${\tt t}$, we get

\begin{equation}
\sigma_k\Big((\x+\y+a)_{a\in \mathbb{F}_q}\Big) = \begin{cases}
    1 & \text{if}\,\, k=0, \\
    -1 & \text{if}\,\, k = q-1, \\
    (\x^q-\x)+(\y^q-\y) & \text{if} \,\,k = q, \\
    0 & \text{otherwise}. 
\end{cases}
\end{equation}

Let $n\geq 0$ be an integer. Then by Waring's formula, we have 
\[
\begin{split}
\sum_{a\in \mathbb{F}_q}(\x+\y+a)^n&=\sum_{\alpha(q-1)+\beta q=n}\,(-1)^{\alpha} \frac{(\alpha+\beta-1)!n}{\alpha!\beta!}(-1)^{\alpha}(\x^q-\x+\y^q-\y)^{\beta}\cr  
&\stackrel{l=\alpha + \beta}{=}\sum_{\frac{n}{q}\leq l \leq \frac{n}{q-1}}\,(-1)^{\alpha} \frac{(l-1)!n}{(lq-n)!(n-l(q-1))!}(-1)^{\alpha}(\x^q-\x+\y^q-\y)^{n-l(q-1)}\cr
&= \sum_{\frac{n}{q}\leq l \leq \frac{n}{q-1}}\,\frac{n}{l}\,\binom{l}{n-l(q-1)}\,(\x^q-\x+\y^q-\y)^{n-l(q-1)}
\end{split}
\]

Set 

$$g_{n,q}(\x,\y)=\sum_{\frac{n}{q}\leq l \leq \frac{n}{q-1}}\,\frac{n}{l}\,\binom{l}{n-l(q-1)}\,(\x+\y)^{n-l(q-1)}\,\,\in \mathbb{Z}[\x,\y]$$

Equivalently, 

$$g_{n,q}(\x,\y)=\sum_{\frac{n}{q}\leq l \leq \frac{n}{q-1}}\,\frac{n}{l}\,\binom{l}{n-l(q-1)}\,(e_1(\x,\y))^{n-l(q-1)}\,\,\in \mathbb{Z}[\x,\y]$$

The coefficients of the polynomial are integers because the coefficients in the Waring's formula are integers. In $\mathbb{F}_q[\x,\y]$, we have 

$$\displaystyle{\sum_{a\in \mathbb{F}_q}(\x+\y+a)^n=g_{n,q}(\x^q-\x,\y^q-\y)}.$$

\begin{rmk}\label{rmk 5.3}
We point out to the reader that the bivariate polynomial $g_{n,q}(\x,\y)$ is the composition of the univariate polynomial $g_{n,q}(\x)$ and the elementary symmetric polynomial $e_1(\x,\y)$:

\[
\begin{split}
g_{n,q}(\x,\y)&=\sum_{\frac{n}{q}\leq l \leq \frac{n}{q-1}}\,\frac{n}{l}\,\binom{l}{n-l(q-1)}\,(e_1(\x,\y))^{n-l(q-1)}\cr
&=\sum_{\frac{n}{q}\leq l \leq \frac{n}{q-1}}\,\frac{n}{l}\,\binom{l}{n-l(q-1)}\,\x^{n-l(q-1)}\circ e_1(\x,\y)\cr
&=g_{n,q}(\x)\circ e_1(\x,\y)
\end{split}
\]
\end{rmk}

For the rest of the section, we will present some results on the polynomial in two variables $g_{n,q}(\x,\y)$ and several variables $g_{n,q}(\x_1,\x_2,\ldots, \x_k)$ that will be useful in Sections 6 \& 7. We skip the proofs of the results as they can be obtained by taking the composition the polynomial $g_{n,q}(\x)$ and the elementary symmetric polynomials $e_1(\x,\y)$ or $e_1(\x_1,\x_2,\ldots, \x_k)$ in the corresponding results in \cite{Hou-FFA-2012}.

First, we will present the explicit expression of the polynomial in several variables $g_{n,q}(\x_1,\x_2,\ldots,\x_k) \in \mathbb{Z}[\x_1,\x_2,\ldots,\x_k]$.

$$g_{n,q}(\x_1,\x_2,\ldots,\x_k)=\sum_{\frac{n}{q}\leq l \leq \frac{n}{q-1}}\,\frac{n}{l}\,\binom{l}{n-l(q-1)}\,(e_1(\x_1,\x_2,\ldots,\x_k))^{n-l(q-1)}\,\,$$

\noindent In $\mathbb{F}_q[\x_1,\x_2,\ldots, \x_k]$, we have 

$$\displaystyle{\sum_{a\in \mathbb{F}_q}(\x_1+\x_2+\ldots+\x_k+a)^n=g_{n,q}(\x_1^q-\x_1,\ldots,\x_k^q-\x_k)}.$$

The recurrence relation of the polynomial $g_{n,q}(\x,\y)$ is given in the following proposition:

\begin{prop}
The polynomial $g_{n,q}(\x,\y)$ satisfies the recurrence relation
\begin{equation*}
\begin{cases}
g_{0,q}= \cdots = g_{q-2,q} = 0,\cr
g_{q-1,q}= -1, \cr
g_{n,q}(\x,\y)=e_1(\x,\y)\,g_{n-q,q}(\x,\y)+g_{n-q+1}(\x,\y), \kern 0.5cm n\geq q.
\end{cases}
\end{equation*}
\end{prop}

We can obtain an immediate generalization to $k$ variables of the above proposition as given below:

\begin{prop}
\noindent The polynomial $g_{n,q}(\x_1,\x_2,\ldots,\x_k)$ satisfies the recurrence relation
\begin{equation*}
\begin{cases}
g_{0,q}= \cdots = g_{q-2,q} = 0,\cr
g_{q-1,q}= -1, \cr
g_{n,q}(\x_1,\x_2,\ldots,\x_k)=e_1(\x_1,\x_2,\ldots,\x_k)\,g_{n-q,q}(\x_1,\x_2,\ldots,\x_k)+g_{n-q+1}(\x_1,\x_2,\ldots,\x_k), \cr\kern 0.5cm n\geq q.
\end{cases}
\end{equation*}
\end{prop}

\subsection{Generating functions of polynomial $g_{n,q}(\x,\y)$ and $g_{n,q}(\x_1,\x_2,\ldots,\x_k)$}

\begin{prop} The generating function of the polynomial $g_{n,q}(\x,\y)$ is given by
$$\displaystyle{\sum_{n\geq 0}\,g_n(\x,\y)\,\z^n\,=\,\frac{-\z^{q-1}}{1-\z^{q-1}-\x\z^q-\y\z^q}}$$
\end{prop}

\begin{prop} The generating function of the polynomial $g_{n,q}(\x_1,\x_2,\ldots,\x_k)$ is given by
$$\displaystyle{\sum_{n\geq 0}\,g_n(\x_1,\ldots,\x_k)\,\z^n\,=\,\frac{-\z^{q-1}}{1-\z^{q-1}-\x_1\z^q-\cdots -\x_k\z^q}}$$
\end{prop}

The following two results, in which the base $q$ weight of $n$ is $q$, are the bivariate and multivariate cases of \cite[Lemma 3.3]{Hou-FFA-2012}. 

\begin{lem}\label{L1}
\noindent Let $ n = \alpha_{0} q^{0} + \cdots + \alpha_{t} q^{t}\hspace{0.2cm},\hspace{0.2cm} 0 \leq \alpha_{i} \leq q - 1 $ and $w_{q}(n)$ be the base $q$ weight of $n$,
\begin{equation}\label{RR20}
g_{n,q}(\x,\y)=
\begin{cases}
0&\text{if}\ w_{q}(n) < q - 1,\cr
-1&\text{if}\ w_{q}(n)  = q - 1,\cr
L(\x)+L(\y) + \delta &\text{if}\ w_{q}(n) = q,
\end{cases}
\end{equation}

\noindent where
$$L(\z)=\alpha_{0}\z^{q^{0}} + ( \alpha_{0} + \alpha_{1} ) \z^{q^{1}} + \cdots  + ( \alpha_{0} + \cdots  + \alpha_{t-1}) \z^{q^{t-1}},$$

\noindent and

\noindent 
\begin{center}
\begin{displaymath}
   \delta = \left\{
     \begin{array}{lr}
       1 &  \text{if} \hspace{0.2cm} q = 2,\cr
       0 &  \text{if} \hspace{0.2cm} q > 2.
     \end{array}
   \right.
\end{displaymath}
\end{center}
\end{lem}

\begin{lem}\label{L6.3}
\begin{equation}
g_{n,q}(\x_1,\x_2,\ldots,\x_k)=
\begin{cases}
0&\text{if}\ w_{q}(n) < q - 1,\cr
-1&\text{if}\ w_{q}(n)  = q - 1,\cr
L_1(\x_1)+L_2(\x_2)+\cdots +L_k(\x_k) + \delta &\text{if}\ w_{q}(n) = q,
\end{cases}
\end{equation}

\noindent where
$$L_i(\x_i)=\alpha_{0}\x_i^{q^{0}} + ( \alpha_{0} + \alpha_{1} ) \x_i^{q^{1}} + \cdots  + ( \alpha_{0} + \cdots  + \alpha_{t-1}) \x_i^{q^{t-1}},\,\,\,\,\text{for}\,\,\, 1\leq i\leq k,$$

\noindent and

\begin{center}
\begin{displaymath}
   \delta = \left\{
     \begin{array}{lr}
       1 &  \text{if} \hspace{0.2cm} q = 2,\cr
       0 &  \text{if} \hspace{0.2cm} q > 2.
     \end{array}
   \right.
\end{displaymath}
\end{center}

\end{lem}

The following two lemmas are due to \cite[Eq. (4.1)]{Hou-FFA-2012} and Remark~\ref{rmk 5.3}.

\begin{lem}\label{L6.4} Let $l,i\geq 0$. Then we have
$$g_{l+q^i}(\x,\y)=g_{l+1}(\x,\y)+S_i(e_1(\x,\y))\cdot g_l(\x,\y).$$
\end{lem}

\begin{lem} A generalization of Lemma~\ref{L6.4} is given by 
$$g_{l+q^i}(\x_1,\ldots, \x_k)=g_{l+1}(\x_1,\ldots, \x_k)+S_i(e_1(\x_1,\ldots, \x_k))\cdot g_l(\x_1,\ldots, \x_k).$$   
\end{lem}

\section{Permutation polynomials in several variables and local permutation polynomials}

Let $\mathbb{F}_q^k$ be the $k$-fold cartesian product of $\mathbb{F}_q$, where $k\geq 1$ is an integer. Let $\overline{\x}=(\x_1,\ldots, \x_k)$, $\overline{\x}_i=(\x_1,\ldots, \x_{i-1},\x_{i+1},\ldots, \x_k)$, and $\mathbb{F}_q[\overline{\x}]$ denote the ring of polynomials in $k$ variables over $\mathbb{F}_q$.  \vskip 0.1in 

We first present some results that will be used throughout the section. 

\begin{thm}(\cite{Lidl-Niederreiter-1972})\label{T2.2}
Suppose $f\in \mathbb{F}_q[\overline{\x}]$ is of the form 
$$f(\bar{\x})=g(\x_1,\ldots, \x_m)+h(\x_{m+1},\ldots, \x_k),\,\,\,\,1\leq m\leq k.$$

\noindent If at least one of $g$ and $h$ is a permutation polynomial over $\mathbb{F}_q$ then $f$ is a permutation polynomial over $\mathbb{F}_q$. If $q$ is prime, then the converse holds.
\end{thm}

\begin{thm}(\cite{GJ})\label{GJ}\label{thm 6.4}
Let $f\in \mathbb{F}_q[\x_1,\ldots, \x_k]$ be of the form 
$$f(\x_1,\ldots,\x_k)=g(\x_1,\ldots,\x_m)+h(\x_{m+1},\ldots,\x_k),\,\,\,\,1\leq m<k.$$
Then $f$ is an LPP if and only if both $g$ and $h$ are local permutation polynomials. 
\end{thm}

\begin{thm}(\cite{GJ})
If $f$ is an LPP, then $f$ is linear if $q=2$ and $q=3$, and has degree at most $n(q-2)$ otherwise.     
\end{thm}


\begin{lem} (\cite[Theorem~3]{GJ}\label{lem 7.6})
 Let $h,\,h_1,\,h_2,\,\dots,\,h_k\in\mathbb{F}_q[\x]$ be permutation polynomials over $\mathbb{F}_q.$ Then for any non constant polynomial $g\in\mathbb{F}_q[\bar{\x}],$
\begin{enumerate}
\item
$g$ is a (local) permutation polynomial if and only if $h(g(\bar{\x}))$ is a (local) permutation polynomial.
\item
$g$ is a (local) permutation polynomial if and only if $g(h_1(\x_1),\,h_2(\x_2),\,\dots,\,h_k(\x_k))$ is a (local) permutation polynomial.
\end{enumerate}
\end{lem}

For the rest of the section, we assume that $Q=q^e$ unless otherwise specified. In the following result, we show the equivalence between the local permutation property of $g_{n,q}(\x, \y)$ and the permutation property of $g_{n,q}(\x)$. 

\begin{prop}\label{prop 7.7}
The polynomial $g_{n,q}(\x,\,\y)$ is an LPP over $\mathbb{F}_{q^e}$ if and only if $g_{n,q}(\x)$ is a PP over $\mathbb{F}_{Q}.$ 
\end{prop}
\begin{proof}
From Remark \ref{rmk 5.3} we have $g_{n,q}(\x,\,\y)=g_{n,q}(\x)\circ e_1(\x,\y).$  Now the necessity follows from Lemma \ref{lem 7.6} and the fact that $e_1(\x,\,\y)=\x+\y$ is an LPP over $\mathbb{F}_{Q}.$

Conversely, suppose that $g_{n,q}(x,\,y)$ is an LPP over $\mathbb{F}_{Q}.$ As $e_1(\x,\,a)$ is a PP over $\mathbb{F}_{Q}$ for each $a\in \mathbb{F}_{Q}$, $g_{n,q}(\x)$ must be a PP over $\mathbb{F}_{Q}.$
\end{proof}

In the following result, we explain the equivalence of the permutation property of $g_{n,q}(\x,\y)$ and $g_{n,q}(\x)$. 

\begin{prop}\label{prop 7.8}
The polynomial $g_{n,q}(\x,\,\y)$ is a permutation polynomial over $\mathbb{F}_{Q}$ if and only if $g_{n,q}(\x)$ is a PP over $\mathbb{F}_{Q}.$ 
\end{prop}
\begin{proof}
Clearly, the necessary part follows from Proposition \ref{prop 7.7}.

 Conversely, assume that $g_{n,q}(\x,\,\y)$ is a permutation polynomial over $\mathbb{F}_{Q}.$ Let $a\in\mathbb{F}_{Q}.$ Let the solutions of the polynomial equation $g_{n,q}(\z)=a$ be $N.$ As $e_1(\x,\,\y)$ is a permutation polynomial, for each $z\in\mathbb{F}_{Q},\, e_1(\x,\,\y)=\z$ has $Q$ solutions in $\mathbb{F}_{Q}^2.$ That is, each $\z$ contributes $Q$ pairs of $(\x,\,\y)$ to the equation $e_1(\x,\,\y)=\z$. Therefore, we have $NQ$ pairs of $(\x,\,\y)\in\mathbb{F}_{Q}^2$ such that $g_{n,q}(e_1(\x,\,\y))=a$ for each $a\in\mathbb{F}_{Q}.$ But it is given that $g_{n,q}(\x,\,\y)$ is a permutation polynomial. Hence, $NQ=Q,$ which implies $N=1.$ Thus we have the desired result. 
\end{proof}

We obtain the following immediate corollary. 

\begin{cor}\label{cor 7.9}
The polynomial $g_{n,q}(\x,\,\y)$ is an  LPP over $\mathbb{F}_{Q}$ if and only if it is a permutation polynomial over $\mathbb{F}_{Q}.$ 
\end{cor}

As mentioned in the introduction, every local permutation polynomial is a permutation polynomial but the converse is not true in general. However, Corollary~\ref{cor 7.9} says that converse also holds in the case of polynomial $g_{n,q}(\x,\,\y)$. 

If we replace $g_{n,q}(\x)$ by any non constant polynomial $f(\x)\in\mathbb{F}_{Q}$ and $g_{n,q}(\x,\,\y)$ by $f(h(\x,\,\y))\in\mathbb{F}_{Q}[\x,\,\y],$ where $h(\x,\,\y)$ is an LPP over $\mathbb{F}_{Q}$, then the Proposition \ref{prop 7.7}, Proposition \ref{prop 7.8} and Corollary \ref{cor 7.9} still hold. We can rewrite them as follows.
\begin{thm}\label{thm 7.10}
Let $f(\x)\in\mathbb{F}_{Q}[x]$ be any non constant polynomial and $g(\x,\,\y)=f(h(\x,\,\y)),$ where $h(\x,\,\y)\in\mathbb{F}_{Q}[\x,\y]$ be any LPP over $\mathbb{F}_{Q}.$ Then the following are equivalent: \begin{enumerate} 
\item $f(\x)$ is a PP over $\mathbb{F}_{Q}$.
\item $g(\x,\,\y)$ is a local permutation polynomial over $\mathbb{F}_{Q}$.
\item $g(\x,\,\y)$ is a permutation polynomial over $\mathbb{F}_{Q}$.
\end{enumerate} 
\end{thm}
\begin{proof}
It is easy to see that $(1)\implies (2)$ follows from Lemma~\ref{lem 7.6}. The implication $(2)\implies (3)$ is obvious, and $(3)\implies (1)$ follows from the proof of second part of Proposition~\ref{prop 7.8}. Hence, all three statements are equivalent.
\end{proof}
We point out to the reader that Theorem \ref{thm 7.10} also holds in multivariate case. 

Since the symmetric polynomial $e_1$ with $k$ variables is an LPP for each $k\geq 2$, studying the local permutation or permutation behaviour of the multivarite $g$-polynomial is equivalent to studying the permutation behaviour of its univariate case. Because the permutation beahviour of the univariate case has well been studied, for the rest of the paper, we will be investigating the permutation and the local permutation behaviour of the polynomial
$$g_{n,q}^{(\ell)}(\x_1,\,\x_2,\dots,\,\x_k)=g_{n,q}(\x)\circ\sigma_{\ell}(\x_1,\,\x_2,\dots,\,\x_k),$$ where $\sigma_{\ell}=\sigma_{\ell}(\x_1,\,\x_2,\dots,\,\x_k)=\x_1^{\ell}+\x_2^{\ell}+\cdots +\x_k^{\ell},$ 
where $\ell\geq 1.$ 

\begin{thm}\label{thm 7.11}
Let $\gcd(\ell,\,Q-1)=1.$ Then $g_{n,q}^{(\ell)}\in\mathbb{F}_{Q}[\bar{\x}]$ is a (local) permutation polynomial over $\mathbb{F}_{Q}$ if and only if $g_{n,q}(\x)$ is a PP over $\mathbb{F}_{Q}.$
\end{thm}
\begin{proof}
We have $\sigma_{\ell}(\x_1,\,\x_2,\dots,\,\x_k)=e_1(\x_1^{\ell},\,\x_2^{\ell},\dots,\,\x_k^{\ell})$ and $e_1$ is always an LPP. Therefore by (2) of Lemma \ref{lem 7.6}, $\sigma_{\ell}$ is a local permutation polynomial when $\gcd(\ell,\,Q-1)=1.$ Thus by Theorem \ref{thm 7.10}, $g_{n,q}^{(\ell)}$ is an LPP over $\mathbb{F}_{Q}$ if and only if $g_{n,q}^{(\ell)}$ is a PP over $\mathbb{F}_{Q}$ if and only if $g_{n,q}(\x)$ is a PP over $\mathbb{F}_{Q}$.
\end{proof}

\begin{rmk}
We point out to the reader that Corollary \ref{cor 7.9} is a particular case of  Theorem \ref{thm 7.11} with $\ell=1$ and $k=2$. Due to Theorem \ref{thm 7.11}, from now on, we study the case where $\gcd(\ell,\,Q-1)\neq 1$. 
\end{rmk}

 \begin{thm}
     Let $\gcd(\ell,\,Q-1)\neq 1.$ Then $g_{n,q}^{(\ell)}$ can never be an LPP over $\mathbb{F}_{Q}.$
 \end{thm}
 \begin{proof}
     It follows from Theorem~\ref{thm 6.4} that, $\sigma_{\ell}$ is not an LPP over $\mathbb{F}_{Q},$ whenever $\gcd(\ell,\,Q-1)\neq 1.$ The proof follows from the fact that the composition of two univariate polynomials is not a PP whenever one of them is not a PP.
 \end{proof}

Next, we will investigate the permutation behaviour of $g_{n,q}^{(\ell)},$ when $\gcd(\ell,\,q^e-1)\neq 1.$

Let $g_{n,q}$ be a PP over $\mathbb{F}_{Q}$ and $N$ be the number of solutions of $g_{n,q}^{(\ell)}(\bar{\x})=a,$ where $a=g_{n,q}(b)$ for some $b\in\mathbb{F}_{Q}.$ As $g_{n,q}^{(\ell)}(\bar{\x})=g_{n,q}(\sigma_{\ell}(\bar{\x}))$ and $g_{n,q}$ is a PP, we must have $\sigma_{\ell}(\bar{\x})=b.$ This means that $\sigma_{\ell}(\bar{\x})=b$ has exactly $N$ solutions. If $N\neq Q^{k-1},$ for some $b\in\mathbb{F}_{Q},$  then $g_{n,q}^{(\ell)}$ is not a PP over $\mathbb{F}_Q$.

Thus, the problem reduces to counting the number of solutions of $\sigma_l(\x_1,\,\x_2,\dots,\,\x_k)=b$, a special class of diagonal equations over $\mathbb{F}_{Q}$ where all coefficients are equal to $1$ and all exponents equal to $\ell$.  A considerable amount of research has been devoted to the study of solutions of diagonal equations over finite fields. As a particular case of \cite[Theorem~6.33, Theorem 6.34]{Lidl-Niederreiter-97}, we have the following result.
\begin{lem}(\cite{Lidl-Niederreiter-97})
   Let $b\in\mathbb{F}_Q.$ The number of solutions $N$ of $\sigma_l(\x_1,\,\x_2,\,\dots,\,\x_k)=b$ in $\mathbb{F}_Q^n$ is given by
$$N=Q^{k-1}+\sum_{i_1=1}^{d-1}\cdots\sum_{i_k=1}^{d-1} J_0(\lambda_1^{i_1},\dots,\lambda_k^{i_k}),$$
   
    where $\lambda_i,\,1\leq i\leq k$ is a multiplicative character of $\mathbb{F}_Q$ of order $d=\gcd(k,\,Q-1)$ and $$J_b(\lambda_1^{i_1},\dots,\lambda_k^{i_k})=\sum_{c_1+\cdots+c_k=b}^{} \lambda_1^{i_1}(c_1)\cdots\lambda_k^{i_k}(c_k)$$ is the Jacobi sum.
\end{lem}

\begin{rmk}\label{rmk 7.14}
Since $N$ does not depend directly on the exponent $\ell,$ but only on the greatest common divisor $d$ of $\ell$ and $Q-1$, we have $N(\sigma_{\ell}=b)=N(\sigma_d=b).$ 
\end{rmk}
\begin{prop}\label{prop 7.15}
    If $\gcd(\ell,\,Q-1)=2,$ then $\sigma_{\ell}\in\mathbb{F}_Q[\x_1,\,\dots,\,\x_k]$ is not a PP over $\mathbb{F}_Q.$
\end{prop}
\begin{proof}
    By the result of Lidl and Niederreiter \cite [p. 284-285]{Lidl-Niederreiter-97}, the following explicit formula is obtained for the number of solutions $N$ of $\sigma_2(\x_1,\,\dots,\,\x_k)=b,$ where $b\in\mathbb{F}_Q^*.$ 
    \[
    \begin{split}
        N&=\begin{cases}
        Q^{k-1}-Q^{\frac{k-2}{2}}\eta((-1)^{\frac{k}{2}})&\text{if}~k~\text{is even},\cr
        Q^{k-1}+Q^{\frac{k-1}{2}}\eta((-1)^{\frac{k-1}{2}}b)&\text{if}~k~\text{is odd},
         \end{cases}
    \end{split}
    \]
    
    where $\eta$ is the quadratic character of $\mathbb{F}_Q.$ As $\eta(c)\neq 0,$ for any $c\in\mathbb{F}_Q^*,$ $N\neq Q^{k-1}.$ Now it follows from Remark~\ref{rmk 7.14} that for any $b\in\mathbb{F}_Q^*,$ the number of solutions of $\sigma_l(\x_1,\,\dots,\,\x_k)=b,$ whenever $\gcd(\ell,\,Q-1)=2$, is never $Q^{k-1}.$ 
\end{proof}

The following definition of $(p,\, r)$-\textit{admissible} integers first appeared in \cite{JAO-2021} to the best of our knowledge. 
 
\begin{defn}
    Let $r>0$ be an integer. An integer $d$ is said to be $(p,\, r)$-\textit{admissible} if $r$ is the smallest integer such that $d\mid (p^r+1)$. 
\end{defn}
\begin{lem}\cite[Theorem 2.14, Remark 7.3]{JAO-2021}\label{lem 7.17}
    Let $N$ denote the number of solutions of $\sigma_d(\x_1,\,\dots,\,\x_k)=0$ in $\mathbb{F}_Q^n,$ where $d>2$ is a divisor of $(Q-1)$ and $k>2.$ Also assume that $(k,\,d)\neq(4,\,3).$ Then the following are equivalent.
    \begin{enumerate}
        \item 
        $|N-Q^{k-1}|=M_k(d)(Q-1)Q^{\frac{k-2}{2}},$ where $M_k(d)$ is the cardinality of the set of all $(j_1,\,j_2,\,\dots,\,j_k)\in\mathbb{Z}^k$ such that $1\leq j_i \leq d-1$ and $j_1+j_2+\cdots+j_k\equiv 0 \pmod d.$
        \item 
        $Q=p^s$, where $s$ is even, and $d$ is $(p,\,r)$- admissible.
    \end{enumerate}
    Moreover, if $k=2$, then $N=Q+(Q-1)(d-1)$. 
\end{lem}
The case $(k,\,d)=(4,\,3)$ is explained in the following lemma that appeared in \cite{JGC-2025}.
\begin{lem}(\cite{JGC-2025})\label{lem 7.18}
Let $Q=p^s,$ where $s$ is even and a divisor $d\geq 3$ of $(Q-1)$ be $(p,\,r)$-admissible. Then the number of roots $N^*$ of $\sigma_d(\x_1,\,\dots,\,\x_k)$ in $(\mathbb{F}_Q^*)^k$ is given by 
\begin{equation}\label{eqn 7.1}
N^*=\frac{1}{Q}\left((Q-1)^k+\frac{Q-1}{d}A(d)^k+\frac{(Q-1)(d-1)}{d}B(d)^k\right),
\end{equation}
where 
\[
\begin{split}
    A(d)&=-1-(-1)^{\frac{s}{2r}}(d-1)Q^{\frac{1}{2}},\cr
    B(d)&=-1+(-1)^{\frac{s}{2r}}Q^{\frac{1}{2}}.
\end{split}  
\]
\end{lem}
\begin{prop}\label{prop 7.19}
    $\sigma_{\ell}(\x_1,\, \x_2)$ is a PP over $\mathbb{F}_Q$ if and only if $\gcd(\ell,\,Q-1)=1.$
\end{prop}
\begin{proof}
    The necessity follows from Theorem~\ref{T2.2}.

    Assume to the contrary that $\gcd(\ell,\,Q-1)=d>1.$ In fact we must have $d>2$ because of Proposition~\ref{prop 7.15}. Lemma~\ref{lem 7.17} says that the number of solutions $N$ of $\sigma_d(x_1,\,x_2)=0$ in $\mathbb{F}_{Q}^2$ is $Q+(Q-1)(d-1)\neq Q.$ The desired result follows from Remark~\ref{rmk 7.14}.
\end{proof}
\begin{rmk}\label{rmk 7.20}
   The following formula for $M_k(d)$ appears in \cite[p. 293]{Lidl-Niederreiter-97}.
$$M_k(d)=(-1)^k+\sum_{r=1}^{k}(-1)^{k-r}\sum_{1\leq i_1<i_2<\cdots<i_r\leq k}^{}\frac{d^r}{\text{lcm}~(\underbrace{d,\,d,\,\dots,\,d}_{r})}.$$
We simplify the expression to obtain
\[
\begin{split}
    M_k(d)&=\begin{cases}
1+\frac{(d-1)^k-1}{d}&\text{if}~k~\text{is even},\cr
-1+\frac{(d-1)^k+1}{d}&\text{if}~k~\text{is odd},
\end{cases}
\end{split}
\]
Clearly, $M_k(d)=0$ if and only if $k$ is odd and $d=2.$ 
\end{rmk}

\begin{prop}\label{prop 7.21}
    Let $Q=p^s,$ where $s\equiv 0\pmod 2$ and $d=\gcd(\ell,\,Q-1)$ be $(p,\,r)$-admissible. Then $\sigma_\ell\in\mathbb{F}_Q[\x_1,\,\dots,\,\x_k],\,k>2$ is never a PP over $\mathbb{F}_Q.$
\end{prop}
\begin{proof}
    Case I:  $(k,d)\neq(4,3).$ 
    
    By Lemma~\ref{lem 7.17}, we have $|N-Q^{k-1}|=M_k(d)(Q-1)Q^{\frac{k-2}{k}},$ where $N$ denotes the number of solutions of $\sigma_d(\x_1,,\,\dots,\,\x_k)=0$ in $\mathbb{F}_Q^k.$ Since $d>2,$  Remark~\ref{rmk 7.20} implies that $M_k(d)\neq 0.$ Consequently, $N\neq Q^{k-1},$ which implies that $\sigma_d$ is  not a PP over $\mathbb{F}_Q$, and therefore the desired result follows from Remark~\ref{rmk 7.14}. 

    Case II: $(k,d)=(4,3).$

    Substituting this value in Equation~\ref{eqn 7.1}, we obtain
    \[
    \begin{aligned}
        N^*(\sigma_d)&=\begin{cases}
            \frac{1}{Q}\left[(Q-1)^4+\frac{Q-1}{3}(1+2Q^{\frac{1}{2}})^4+\frac{2(Q-1)}{3}(Q^\frac{1}{2}-1)^4\right] & \text{if}~\frac{s}{2r}~\text{is even},\cr
            \frac{1}{Q}\left[(Q-1)^4+\frac{Q-1}{3}(-1+2Q^{\frac{1}{2}})^4+\frac{2(Q-1)}{3}(Q^\frac{1}{2}+1)^4\right] & \text{if}~\frac{s}{2r}~\text{is odd},
        \end{cases}
        \\[2 mm]
        &=\begin{cases}
            Q^3+C_1 & \text{if}~\frac{s}{2r}~\text{is even},\cr
              Q^3+C_2 & \text{if}~\frac{s}{2r}~\text{is odd},
        \end{cases}
        \qquad \text{where}
    \end{aligned}
    \]
    \[
    \begin{aligned}
        C_1&=\frac{-4Q^3+6Q^2-4Q+1}{Q}+\frac{Q-1}{3Q}(1+2Q^{\frac{1}{2}})^4+\frac{2(Q-1)}{3Q}(Q^{\frac{1}{2}}-1)^4,\\
    C_2&=\frac{-4Q^3+6Q^2-4Q+1}{Q}+\frac{Q-1}{3Q}(-1+2Q^{\frac{1}{2}})^4+\frac{2(Q-1)}{3Q}(Q^{\frac{1}{2}}+1)^4,
    \end{aligned}
    \]
     and $N^*(\sigma_d)$ denotes the number of solutions of $\sigma_d(x_1,\,\dots,\,x_k)=0$ in $(\mathbb{F}_Q^k)^*.$
    We claim that $C_1+1\neq 0$ and $C_2+1\neq 0.$ Let $Q^{1/2}=t,$ where $t$ is a positive integer. Then after simplification, we have 
    \[
    \begin{aligned}
        C_1&=2t^4+8t^3+12t^2-8t-15,\,\text{and}\\
    C_2&=2t^4-8t^3+12t^2+8t-15.
    \end{aligned}
    \]
     Now $C_1+1=0$ if and only if $g_1(t)=2t^4+8t^3+12t^2-8t-14=0$ and $C_2+1=0$ if and only if $g_2(t)=2t^4-8t^3+12t^2+8t-14=0.$ Since $t\in\mathbb{Z}^+$, the Rational Root Theorem implies that, any integer root of $g_1(t)$ or $g_2(t)$ must divide $14.$ Hence, it suffices to check for only $t=2$ and $t=7$ since $t=Q^{1/2}$ and $Q$ is a square of a prime. A direct computation shows that $g_i(2)\neq 0$ and $g_i(7)\neq 0$ for any $i\in\{1,\,2\}.$

    Thus, $N=N^*(\sigma_d)+1\neq Q^3$ and it follows that $\sigma_d$ is not a PP over $\mathbb{F}_Q,$ and therefore $\sigma_\ell$ is not a PP over $\mathbb{F}_Q$ by Remark~\ref{rmk 7.14}.
\end{proof}

Combining Proposition~\ref{prop 7.15}, Proposition~\ref{prop 7.19} and Proposition~\ref{prop 7.21}, we obtain the following theorem.
\begin{thm}
     Let $q=p^m,\,Q=p^e,\,g_{n,q}$ be a PP over $\mathbb{F}_Q$ and assume that $\gcd(\ell,\,Q-1)=d>1.$ Then the following statements hold.
    \begin{enumerate}
        \item 
        $g_{n,q}^{(\ell)}(\x_1,\,\x_2)\in\mathbb{F}_Q[\x_1,\,\x_2]$ is not a PP over $\mathbb{F}_Q.$
        \item 
        If $d=2,$ then $g_{n,q}^{(\ell)}$ is not a PP over $\mathbb{F}_Q.$ 
        \item 
        If $d>2$ is $(p,\,r)$-admissible and $me$ is even then $g_{n,q}^{(\ell)}$ is not a PP over $\mathbb{F}_Q.$ 
        
    \end{enumerate}
\end{thm}

\begin{rmk}
    If $g_{n,q}$ is not a PP over $\mathbb{F}_Q,$ then there exists $a\in\mathbb{F}_Q$ such that $a\neq g_{n,q}(b)$ for any $b\in\mathbb{F}_Q.$ If $g_{n,q}^{(\ell)}$ is a PP over $\mathbb{F}_Q,$ then $g_{n,q}^{(\ell)}=a$ has exactly $Q^{k-1}$ solutions in $\mathbb{F}_Q^k.$ It implies that $g_{n,q}(\sigma_\ell(\x_1,\,\dots,\,\x_k))=a$ has exactly $Q^{k-1}$ solutions in $\mathbb{F}_Q^k.$ Consequently, we can write $a=g_{n,q}(b)$ for some $b\in\mathbb{F}_Q,$ which is a contradiction.

Therefore, whenever univariate $g$-polynomial is not a PP, the multivariate $g$-polynomial $g^{(\ell)},\,\ell\geq 1$ can never be a PP irrespective of the gcd condition.
\end{rmk}
   
\begin{conj}
 Let $q=p^m$  and $Q=p^e.$ The polynomial $g_{n,q}^{(\ell)}\in\mathbb{F}_Q[\x_1,\,\dots,\,\x_k]$ is a PP over $\mathbb{F}_Q$ if and only if $\gcd(\ell,\,Q-1)=1$ and $g_{n,q}(\x)$ is a PP over $\mathbb{F}_Q.$
\end{conj}

\begin{rmk}To prove the conjecture one only needs to show that if $\sigma_\ell\in\mathbb{F}_Q[\x_1,\,\dots,\,\x_k]$ is a PP over $\mathbb{F}_Q,$ then $\gcd(\ell,\,Q-1)=1$.
\end{rmk} 

\section{Some open questions on PPs from polynomial $g_{n,q}(\x)$}

In this section, we present some open questions on the permutation behaviour of the polynomial $g_{n,q}$ over finite fields of even characteristic. In Table 1, we have pointed out the unexplained desirable triples that each of the following open questions, once answered, would explain. 

\begin{ques}\label{Q1}
Let $q>2$ be even, $e>0$ be even, and $n=(q-1)\cdot q^0+(q-1)\cdot q^{e/2}+1\cdot q^e$. Then $g_{n,q}=S_1+S_{e/2}^{q^{e/2}}+S_e\cdot S_{e/2}^{q-1}$. Find conditions on $q$ and $e$ for which  $g_{n,q}$ is a PP of $\mathbb{F}_{q^e}.$  
\end{ques}


\begin{ques}\label{Q2}
Let $q=4$, $n=1\cdot q^0+1\cdot q^{e-2}+1\cdot q^{e-1}+2\cdot q^e$, and $e>2$. Then $$g_{1\cdot q^0+1\cdot q^{e-2}+1\cdot q^{e-1}+2\cdot q^e}=S_{e-1}\cdot S_e+S_2^{q^{e-2}}\cdot S_{e-1}+S_e^2.$$ Find conditions on $e$ for which $g_{n,q}$ is a PP of $\mathbb{F}_{q^e}.$  
\end{ques}


\begin{ques}\label{Q3}
Let $q>2$ be even, $n=1\cdot q^0+(q-1)\cdot q^{e/2}+(q-1)\cdot q^e$, and $e>0$ be even. Then 
$$g_{1\cdot q^0+(q-1)\cdot q^{e/2}+(q-1)\cdot q^e}=S_{e/2+1}^{q^{e/2}}+S_e+S_e\,(S_{e/2}^{q^{e/2}})^{q-1}.$$
Find conditions on $q$ and $e$ for which $g_{n,q}$ is a PP of $\mathbb{F}_{q^e}.$  
\end{ques}


\begin{ques}\label{Q4}
Let $q>2$ be even and $e>2$ a positive integer. Let $n=(q-1)\cdot q^0+(q-2)\cdot q^2+1\cdot q^e+1\cdot q^{e+2}$. Find conditions on $q$ and $e$ for which $$g_{n,q}=\x^q+S_e^2\,S_2^{q-2}+S_eS_2^{q-1}$$ is a PP of $\mathbb{F}_{q^e}.$  
\end{ques}


\begin{ques}\label{Q5}
Let $q>2$ be even and $e>3$ a positive integer. Let $n=(q-1)\cdot q^0+(q-2)\cdot q^2+1\cdot q^{e-1}+1\cdot q^{e+1}$. Find conditions on $q$ and $e$ for which $$g_{n,q}=\x^{q^{e-1}}+\x\,S_2^{q-2}S_{e-1}+S_2^{q-2}\,S_{e-1}\,S_e$$ is a PP of $\mathbb{F}_{q^e}.$  
\end{ques} 


\begin{ques}\label{Q6}
Let $q=4$ and $e>2$ be even. Let $$n=1\cdot q^0+2\cdot q^{\frac{e}{2}-1}+2\cdot q^{e-1}+1\cdot q^{e}+2\cdot q^{\frac{3e-2}{2}}.$$ Find conditions on $e$ for which \[
\begin{split}
g_{1\cdot q^0+2\cdot q^{\frac{e}{2}-1}+2\cdot q^{e-1}+1\cdot q^{e}+2\cdot q^{\frac{3e-2}{2}}}&=S_{\frac{e}{2}-1}^2+S_{\frac{e}{2}}^{2\cdot q^{e-1}}+S_e\,S_{\frac{3e}{2}-1}\,+\,\x\,S_e+S_e\,S_{e-1}^2\,S_{\frac{e}{2}-1}^2\cr
&+S_e\,S_{\frac{3e}{2}-1}+S_e\,S_{\frac{3e}{2}-1}^2\,\Big(S_{\frac{e}{2}-1}^2+S_{e-1}^2\Big)
\end{split}
\] is a PP of $\mathbb{F}_{q^e}.$  
\end{ques} 


\begin{ques}\label{Q7}
Let $q=4$, $n=1\cdot q^0+3\cdot q^1+3\cdot q^2+2\cdot q^3+1\cdot q^e+1\cdot q^{e+1}$, and $e\geq 4$. Then 
\[
\begin{split}
g_{n,q}&=(S_2^q\,S_3+\x\,S_3+S_e\,S_{e+1})\,\Big(1+(1+\x^{q-1})\,S_2^{q-1}\Big)+\x\,S_3^2\,\Big((1+\x^{q-1})\,S_2^2+\x^2\,S_2^{q-1}\Big)\cr
&+S_3^2\,S_e\,S_{e+1}\,\Big((1+\x^{q-1})\,S_2+\x^{q+1}\,S_2^2\Big)\cr
\end{split}
\]
Find conditions on $e$ for which $g_{n,q}$ is a PP of $\mathbb{F}_{q^e}.$  
\end{ques} 


\begin{ques}\label{Q8}
Let $q=4$, $n=1\cdot q^0+2\cdot q^1+1\cdot q^2+2\cdot q^3+2\cdot q^{e+1}$, and $e\geq 3$. Then 
$$g_{n,q}=\x S_3+S_2^q\cdot S_3+S_2^{q+1}+S_2\,S_3+S_e^2+\x^qS_2+\x^2S_2S_e^2+S_2S_3^2S_e^2.$$
Find conditions on $e$ for which $g_{n,q}$ is a PP of $\mathbb{F}_{q^e}.$  
\end{ques}  


\section*{Acknowledgments}
The authors are grateful to Allen Broughton, Daniel Katz, Xiang-dong Hou, and Dhiren Kumar Basnet for the valuable discussions and suggestions on a preliminary draft of the paper. Bhitali Kousik was supported by DST-INSPIRE Fellowship, Government of India (INSPIRE Reg. No. IF230368).

%
%

\section*{Appendix}

The following is an update of \cite[Table~1]{Fernando-Hou-FFA-2015} of all desirable triples $(n,e;4)$ with $e\le 6$ and $w_4(n)>4$, where $w_4(n)$ is the base $4$ weight of $n$. Only a few entries remain to be explained.

\begin{table}[!h]
\caption{Desirable triples $(n,e;4)$, $e\le 6$, $w_4(n)>4$}\label{Tb3}
\vspace{-5mm}
\[
\begin{tabular}{|c|r|l|c|}
\hline
$e$ & $n \hfil$ & base $4$ digits of $n$ & reference \\ \hline
\hline

2&59&{3,2,3}&  \cite[Theorem 5.9 (ii)]{Fernando-Hou-Lappano-FFA-2013}    \\  \hline
2&127&{3,3,3,1}&  \cite[Proposition 3.1]{Hou-FFA-2012}  \\  \hline
3&29&{1,3,1}&  \cite[Example 6.3]{Fernando-Hou-Lappano-FFA-2013}    \\  \hline
3&101&{1,1,2,1}&   \cite[Theorem 6.12 ]{Fernando-Hou-Lappano-FFA-2013}   \\  \hline
3&149&{1,1,1,2}&  \textcolor{blue}{Question~\ref{Q2}}    \\  \hline
3&163&{3,0,2,2}&   \cite[Theorem 6.10]{Fernando-Hou-Lappano-FFA-2013}    \\  \hline
3&281&{1,2,1,0,1}&  \cite[Corollary 6.16]{Fernando-Hou-Lappano-FFA-2013}  \\  \hline
3&307&{3,0,3,0,1}&   \cite[Theorem 1.1]{Hou13}  \\  \hline
3&329&{1,2,0,1,1}&  \cite[Example 6.4]{Fernando-Hou-Lappano-FFA-2013}    \\  \hline
3&341&{1,1,1,1,1}&  \cite[Example 6.4 ]{Fernando-Hou-Lappano-FFA-2013}    \\  \hline
3&2047&{3,3,3,3,3,1}&  \cite[Proposition 3.1]{Hou-FFA-2012}    \\  \hline
4&281&{1,2,1,0,1}&     \cite[Theorem 6.12]{Fernando-Hou-Lappano-FFA-2013}   \\  \hline
4&307&{3,0,3,0,1}&   \textcolor{blue}{Question~\ref{Q1}}  \\  \hline
4&401&{1,0,1,2,1}&   \cite[Theorem 6.12]{Fernando-Hou-Lappano-FFA-2013}   \\  \hline
4&547&{3,0,2,0,2}&   \cite[Theorem 6.10]{Fernando-Hou-Lappano-FFA-2013}  \\  \hline
4&779&{3,2,0,0,3}&   \cite[Theorem 6.6]{Fernando-Hou-Lappano-FFA-2013}    \\  \hline
4&787&{3,0,1,0,3}&   \cite[Theorem 6.8]{Fernando-Hou-Lappano-FFA-2013}    \\  \hline
4&817&{1,0,3,0,3}&  \textcolor{blue}{Question~\ref{Q3}}    \\  \hline
4&899&{3,0,0,2,3}&   \cite[Theorem 6.6]{Fernando-Hou-Lappano-FFA-2013}    \\  \hline
4&1469&{1,3,3,2,1,1}& \textcolor{blue}{Question~\ref{Q7}}     \\  \hline
4&2201&{1,2,1,2,0,2}&  \textcolor{blue}{Question~\ref{Q8}}    \\  \hline
4&2317&{1,3,0,0,1,2}&   \cite[Proposition 3.4]{Fernando-Hou-FFA-2015}  \\  \hline
4&2321&{1,0,1,0,1,2}&   \cite[Theorem 6.12]{Fernando-Hou-Lappano-FFA-2013}   \\  \hline
4&2377&{1,2,0,1,1,2}&   \cite[Proposition 3.6]{Fernando-Hou-FFA-2015} \\  \hline
4&2441&{1,2,0,2,1,2}&  \textcolor{blue}{Question~\ref{Q6}}    \\  \hline
4&4387&{3,0,2,0,1,0,1}&  \textcolor{blue}{Question~\ref{Q4}}    \\  \hline
4&32767&{3,3,3,3,3,3,3,1}&  \cite[Proposition 3.1]{Hou-FFA-2012}  \\  \hline
5&29&{1,3,1}&   \cite[ Example 6.3]{Fernando-Hou-Lappano-FFA-2013}  \\  \hline
5&1049&{1,2,1,0,0,1}&   \cite[Theorem 6.12]{Fernando-Hou-Lappano-FFA-2013}    \\  \hline
5&1061&{1,1,2,0,0,1}&  \cite[Theorem 6.12]{Fernando-Hou-Lappano-FFA-2013}    \\  \hline
5&1169&{1,0,1,2,0,1}&  \cite[Theorem 6.12]{Fernando-Hou-Lappano-FFA-2013}   \\  \hline
5&1289&{1,2,0,0,1,1}&  \cite[Theorem 6.12]{Fernando-Hou-Lappano-FFA-2013}    \\  \hline
5&1409&{1,0,0,2,1,1}&  \cite[Theorem 6.12]{Fernando-Hou-Lappano-FFA-2013}    \\  \hline
5&1541&{1,1,0,0,2,1}&   \cite[Theorem 6.12]{Fernando-Hou-Lappano-FFA-2013}  \\  \hline
5&1601&{1,0,0,1,2,1}&  \cite[Theorem 6.12]{Fernando-Hou-Lappano-FFA-2013}   \\  \hline
5&2083&{3,0,2,0,0,2}&  \cite[Theorem 6.10]{Fernando-Hou-Lappano-FFA-2013}     \\  \hline
5&2563&{3,0,0,0,2,2}&   \cite[Theorem 6.9]{Fernando-Hou-Lappano-FFA-2013}    \\  \hline

\end{tabular}
\]
\end{table}

\clearpage

\addtocounter{table}{-1}
\begin{table}[!h]
\caption{continued}
\vspace{-5mm}
\[
\begin{tabular}{|c|r|l|c|}
\hline
$e$ & $n \hfil$ & base $4$ digits of $n$ & reference \\ \hline
\hline

5&4229&{1,1,0,2,0,0,1}&   \cite[Theorem 6.12]{Fernando-Hou-Lappano-FFA-2013}    \\  \hline
5&4289&{1,0,0,3,0,0,1}&  \textcolor{blue}{Theorem~\ref{T4.5}}   \\  \hline
5&4387&{3,0,2,0,1,0,1}&   \textcolor{blue}{Question~\ref{Q5}}   \\  \hline
5&5129&{1,2,0,0,0,1,1}&   \cite[ Example 6.4]{Fernando-Hou-Lappano-FFA-2013}   \\  \hline
5&5141&{1,1,1,0,0,1,1}&   \cite[Example 6.4 ]{Fernando-Hou-Lappano-FFA-2013}   \\  \hline
5&5189&{1,1,0,1,0,1,1}&   \cite[Example 6.4]{Fernando-Hou-Lappano-FFA-2013}    \\  \hline
5&5249&{1,0,0,2,0,1,1}&   \cite[Theorem 6.12]{Fernando-Hou-Lappano-FFA-2013}   \\  \hline
5&5381&{1,1,0,0,1,1,1}&   \cite[Example 6.4]{Fernando-Hou-Lappano-FFA-2013}    \\  \hline
5&8713&{1,2,0,0,2,0,2}&   \cite[Proposition 3.5]{Fernando-Hou-FFA-2015} \\  \hline
5&9281&{1,0,0,1,0,1,2}&   \cite[ Theorem 6.12 ]{Fernando-Hou-Lappano-FFA-2013}  \\  \hline
5&17429&{1,1,1,0,0,1,0,1}&   \cite[Proposition 3.3]{Fernando-Hou-FFA-2015} \\  \hline
5&17441&{1,0,2,0,0,1,0,1}&   \cite[Theorem 6.12]{Fernando-Hou-Lappano-FFA-2013}    \\  \hline
5&17489&{1,0,1,1,0,1,0,1}&   \cite[Proposition 3.3]{Fernando-Hou-FFA-2015} \\  \hline
5&17681&{1,0,1,0,1,1,0,1}&   \cite[Proposition 3.3]{Fernando-Hou-FFA-2015} \\  \hline
5&524287&{3,3,3,3,3,3,3,3,3,1}&  \cite[Proposition 3.1]{Hou-FFA-2012}    \\  \hline

6&4361&{1,2,0,0,1,0,1}&   \cite[Theorem 6.12]{Fernando-Hou-Lappano-FFA-2013}   \\  \hline
6&6161&{1,0,1,0,0,2,1}&    \cite[Theorem 6.12]{Fernando-Hou-Lappano-FFA-2013}    \\  \hline
6&6401&{1,0,0,0,1,2,1}&    \cite[Theorem 6.12]{Fernando-Hou-Lappano-FFA-2013}   \\  \hline
6&8227&{3,0,2,0,0,0,2}&   \cite[Theorem 6.10]{Fernando-Hou-Lappano-FFA-2013}   \\  \hline
6&8707&{3,0,0,0,2,0,2}&   \cite[Theorem 6.11]{Fernando-Hou-Lappano-FFA-2013}   \\  \hline
6&12299&{3,2,0,0,0,0,3}&   \cite[Theorem 6.6]{Fernando-Hou-Lappano-FFA-2013}   \\  \hline
6&12307&{3,0,1,0,0,0,3}&   \cite[Theorem 6.8 ]{Fernando-Hou-Lappano-FFA-2013}   \\  \hline
6&14339&{3,0,0,0,0,2,3}&   \cite[Theorem 6.6]{Fernando-Hou-Lappano-FFA-2013}    \\  \hline
6&37121&{1,0,0,0,1,0,1,2}&   \cite[Theorem 6.12]{Fernando-Hou-Lappano-FFA-2013}    \\  \hline
6&65801&{1,2,0,0,1,0,0,0,1}&   \cite[Corollary 6.16]{Fernando-Hou-Lappano-FFA-2013}    \\  \hline
6&65921&{1,0,0,2,1,0,0,0,1}&  \cite[Theorem 1]{Fernando-2019}   \\  \hline
6&66307&{3,0,0,0,3,0,0,0,1}&    \cite[Theorem 1.1]{Hou13}  \\  \hline
6&135209&{1,2,2,0,0,0,1,0,2}&  \textcolor{blue}{Theorem~\ref{T3.25}}    \\  \hline
6&135217&{1,0,3,0,0,0,1,0,2}&  \cite[Proposition 3.4]{Fernando-Hou-FFA-2015} \\  \hline
6&135457&{1,0,2,0,1,0,1,0,2}&  \cite[Proposition 3.7]{Fernando-Hou-FFA-2015} \\  \hline
6&137249&{1,0,2,0,0,2,1,0,2}&  \textcolor{blue}{Theorem~\ref{T3.27}}   \\  \hline
6&8388607&{3,3,3,3,3,3,3,3,3,3,3,1}&  \cite[Proposition 3.1]{Hou-FFA-2012}    \\  \hline

\end{tabular}
\]
\end{table}

\end{document}